\documentclass{amsart}
\usepackage{amssymb}
\usepackage{braket}
\usepackage{mathrsfs}
\usepackage{amsrefs}
\usepackage{ifthen}
\usepackage{graphicx}

\usepackage{tikz}
\usetikzlibrary{patterns}
\usepackage{comment} 
\usepackage{mleftright}
\usepackage{hyperref}
\usepackage{cleveref}
\usepackage[all]{xy}

\theoremstyle{plain}
\newtheorem{theo}{Theorem}[section]
\crefname{theo}{Theorem}{Theorems}
\Crefname{theo}{Theorem}{Theorems}
\newtheorem{prop}[theo]{Proposition}
\crefname{prop}{Proposition}{Propositions}
\Crefname{prop}{Proposition}{Propositions}
\newtheorem{lem}[theo]{Lemma}
\crefname{lem}{Lemma}{Lemmas}
\Crefname{lem}{Lemma}{Lemmas}
\newtheorem{cor}[theo]{Corollary}
\crefname{cor}{Corollary}{Corollaries}
\Crefname{cor}{Corollary}{Corollaries}

\crefname{claim}{Claim}{Claims}
\Crefname{claim}{Claim}{Claims}

\crefname{property}{Property}{Properties}
\Crefname{property}{Property}{Properties}
\newtheorem{problem}[theo]{Problem}
\crefname{problem}{Problem}{Problems}
\Crefname{problem}{Problem}{Problems}

\theoremstyle{definition}

\crefname{defi}{Definition}{Definitions}
\Crefname{defi}{Definition}{Definitions}

\crefname{notation}{Notation}{Notations}
\Crefname{notation}{Notation}{Notations}

\crefname{convention}{Convention}{Conventions}
\Crefname{convention}{Convention}{Conventions}

\crefname{cond}{Condition}{Conditions}
\Crefname{cond}{Condition}{Conditions}

\crefname{assum}{Assumption}{Assumptions}
\Crefname{assum}{Assumption}{Assumptions}

\theoremstyle{remark}
\newtheorem{rem}[theo]{Remark}
\crefname{rem}{Remark}{Remarks}
\Crefname{rem}{Remark}{Remarks}
\newtheorem{ex}[theo]{Example}
\crefname{ex}{Example}{Examples}
\Crefname{ex}{Example}{Examples}

\crefname{section}{Section}{Sections}
\Crefname{section}{Section}{Sections}
\crefname{subsection}{Subsection}{Subsections}
\Crefname{subsection}{Subsection}{Subsections}
\crefname{figure}{Figure}{Figures}
\Crefname{figure}{Figure}{Figures}

\newtheorem*{acknowledgement}{Acknowledgement}

\newcommand{\Z}{\mathbb{Z}}
\newcommand{\R}{\mathbb{R}}

\newcommand{\CP}{\mathbb{CP}}
\newcommand{\spc}{\mathrm{spin}^c}
\newcommand{\spcGL}{\mathrm{spin}^c_{GL}}
\newcommand{\tGL}{\widetilde{GL}}
\newcommand{\FrGL}{\mathrm{Fr}_{GL}}

\newcommand{\sign}{\mathrm{sign}}

\newcommand{\gauge}{{\mathcal G}}

\newcommand{\calM}{{\mathcal M}}

\newcommand{\calW}{{\mathcal W}}

\newcommand{\fraks}{\mathfrak{s}}
\newcommand{\frakt}{\mathfrak{t}}

\newcommand{\Met}{\mathrm{Met}}
\newcommand{\circPi}{\mathring{\Pi}}
\newcommand{\PSC}{\mathrm{PSC}}

\newcommand{\Diff}{\mathrm{Diff}}
\newcommand{\Aut}{\mathrm{Aut}}

\newcommand{\Map}{\mathrm{Map}}
\newcommand{\Si}{\Sigma}

\newcommand{\vp}{\varphi}
\newcommand{\inc}{\hookrightarrow}

\newcommand{\del}{\partial}

\newcommand{\im}{\mathop{\mathrm{Im}}\nolimits}

\newcommand{\SWinv}{\mathop{\mathrm{SW}}\nolimits}

\newcommand{\SW}{Seiberg--Witten }

\newcommand{\id}{\mathrm{id}}

\title[Positive scalar curvature and families of Seiberg-Witten equations]{Positive scalar curvature and higher-dimensional families of Seiberg-Witten equations}

\author{Hokuto Konno}

\address{RIKEN iTHEMS, Wako, Saitama 351-0198, Japan}

\email{hokuto.konno@riken.jp}

\date{}

\begin{document}

\maketitle

\begin{abstract}
We introduce an invariant of tuples of commuting diffeomorphisms on a $4$-manifold using families of \SW equations.
This is a generalization of Ruberman's invariant of diffeomorphisms defined using $1$-parameter families of \SW equations.
Our invariant yields an application to the homotopy groups of the space of positive scalar curvature metrics on a $4$-manifold.
We also study the extension problem for families of $4$-manifolds using our invariant.
\end{abstract}

\tableofcontents

\section{Introduction}

Ruberman~\cite{MR1671187, MR1734421, MR1874146} has introduced gauge theoretic invariants of a diffeomorphism on a $4$-manifold.
Using one of his invariants in \cite{MR1874146}, he has shown that there exist $4$-manifolds for which the spaces of metrics with positive scalar curvature (PSC for short) are disconnected.
This is the first result on the homotopy groups of the space of PSC metrics on a $4$-manifold.
The main ingredients of these invariants are $1$-parameter families of Yang--Mills ASD or \SW equations.
It is natural to ask whether we can consider a kind of generalization of such an invariant using gauge theory for higher-dimensional families.
In particular, an interesting question is how we may apply such an extended invariant to the topological study of the space of PSC metrics via higher-dimensional families.

In this paper we generalize Ruberman's invariant given in~\cite{MR1671187} using higher-dimensional families of \SW equations.
We shall define an invariant of tuples of commuting diffeomorphisms preserving a given $\spc$ structure on a $4$-manifold.
An important point is that we can give such diffeomorphisms for which our invariant does not vanish and this non-vanishing result yields a new application to the space of PSC metrics on a $4$-manifold.
Let us describe our main application here.
We consider $4$-manifolds obtained as the connected sum of some copies of $\CP^{2}$ and $-\CP^{2}$, which are typical $4$-manifolds admitting PSC metrics.
Ruberman~\cite{MR1874146} has proved that $\pi_{0}(\PSC(X)) \neq 0$ for $X = 2k\CP^{2} \# l(-\CP^{2})$ with $k \geq 2$ and a sufficiently large $l$.
For other numbers of the connected summands, the result due to Xu~\cite{MR2706507}, based on a cohomotopy refinement of Ruberman's invariant, gives the non-triviality of $\pi_{0}(\PSC(X))$ for $X = (4k+7)\CP^{2} \# l(-\CP^{2})$ with $k \geq 0$ and a sufficiently large $l$.
We shall show the following \lcnamecref{application to PSC} in this paper.

\begin{theo}
\label{application to PSC}
Let $k \geq 2$, $n \geq 1$, $l \geq 10k + 2n -1$ be natural numbers and $X$ be the $4$-manifold given by
\[
X = (2k + n - 1)\CP^{2} \# l(-\CP^{2}).
\]
Then, 
\[
\pi_{i}(\PSC(X)) \neq 0
\]
holds for at least one $i \in \{0, \ldots, n-1\}$.
\end{theo}

Setting $n=1$ in \cref{application to PSC} recovers Ruberman's result in \cite{MR1874146} on the disconnectivity of the space of PSC metrics.
To exhibit a new constraint on PSC metrics, let us consider the case that $k = n = 2$ in \cref{application to PSC} for example.
Then we deduce that, for each $l \geq 23$,
\[
\pi_{i}(\PSC((5\CP^{2})\#l(-\CP^{2}))) \neq 0
\]
holds for $i =0$ or $i=1$.
This does not follow from Ruberman's result and Xu's.
 In fact, as explained in \cref{rem: comparison with Ruberman and Xu}, \cref{application to PSC} provides new constraints on $\PSC(X)$ for infinitely many $4$-manifolds $X$'s having distinct $b^{+}(X)$, where $b^{+}(X)$ is the dimension of a maximal positive-definite subspace of $H^2(X;\R)$ with respect to the intersection form of $X$.
(A detailed comparison between \cref{application to PSC} and Ruberman's result and Xu's is given in \cref{rem: comparison with Ruberman and Xu}.)

The following two tools are used to prove \cref{application to PSC}:
the first one is the combination of wall-crossing and gluing technique due to Ruberman~\cite{MR1671187, MR1734421, MR1874146}, and the second is the description of higher-dimensional wall-crossing phenomena in terms of embedded surfaces given in \cite{Konno1} by the author.
Here let us explain the term ``higher-dimensional wall-crossing'';
it also describes the root of \cref{application to PSC}.
Let us start $4$-manifold $X$ with $b^{+}(X)=0$.
For such a $4$-manifold, the celebrated theorem due to Donaldson~\cite{MR710056} tells that reducible points in the moduli space of solutions to the Yang--Mills ASD equation give a strong constraint on the topology of $X$.
This story is valid also in the \SW theory.
In the case that $b^{+}(X) = 1$, the effect of reducible solutions in the moduli space is described as wall-crossing.
This can be regarded as a $1$-parameter analogue of Donaldson's theorem.
Namely, one can find a reducible solution using a suitable $1$-parameter family of ASD/\SW equations and it gives information about the parameterized moduli space.
``Higher-dimensional wall-crossing'' is an analogue of it for $4$-manifolds with general $b^{+} \geq 1$.
Although the effect of reducible solutions weakens for larger $b^{+}$, one can detect it using a $b^{+}$-dimensional family of equations.
The statement of \cref{application to PSC} reflects how weak the effect is:
Ruberman has used usual (i.e. $b^{+}=1$) wall-crossing, and he has proved that $\pi_{0}(\PSC(X)) \neq 0$ for some $X$.
(Although in Ruberman's case the $4$-manifold $X$ he considered has arbitrary large (even) $b^{+}$, his argument is a combination of wall-crossing for a manifold $N$ whose $b^{+}$ is $1$ and a gluing argument with another $4$-manifold $M$.
We note that most $b^{+}$ of $X$ comes from $M$ except for $1$ coming from $N$.)
On the other hand, we use wall-crossing for $n$-dimensional families for general $n \geq 1$ on a $4$-manifold $N$ with $b^{+}(N)=n$, and we can show that at least one of $\pi_{0}(\PSC(X)), \ldots, \pi_{n-1}(\PSC(X))$ is non-trivial for some $X$ obtained as the connected sum of $N$ with another $M$.

We also use our invariant to study the group of diffeomorphisms on a given $4$-manifold preserving a given $\spc$ structure, in particular to study the extension problem for $4$-manifold bundles having this group as the structure group.
For a given fiber bundle with certain structure group on some base space $M$ and for a space $W$ with $M \subset W$, it is a fundamental question whether one can extend the bundle to a bundle over $W$ having the same structure group.
We can use the non-vanishing theorem for our invariant to give an obstruction to extensions of families of $4$-manifolds whose structure group is the group of diffeomorphisms preserving a $\spc$ structure.
(See \cref{extension problem at the level of Diff}.)
We hope that this kind of obstruction might be useful to approach the study of higher-dimensional manifolds via $4$-dimensional gauge theory.

\begin{acknowledgement}
The author would like to express his deep gratitude to Mikio Furuta for the helpful suggestions and for continuous encouragement during this work.
The author would also like to express his appreciation to Nobuhiro Nakamura for stimulating discussion and informing him of Nakamura's work and Ruberman's on gauge theory for families.
Their work is the starting point of the current paper.
The author also wishes to thank Tomohiro Asano and Masaki Taniguchi for useful discussion in a conference held in summer 2017 at UCLA.
In particular, Asano explained to the author an example in \cref{rem extension and ex for surface bundle}.
Finally, the author gratefully acknowledges the many helpful suggestions of the anonymous referees.
The author was supported by JSPS KAKENHI Grant Number 16J05569 and
the Program for Leading Graduate Schools, MEXT, Japan.
\end{acknowledgement}

\section{Invariant of tuples of diffeomorphisms}

In this section, for a given $\spc$ $4$-manifold, we define an invariant of $n$-tuples ($n \geq 1$) of  commuting diffeomorphisms preserving the $\spc$ structure.
This is a higher-dimensional analogue of the \SW invariant of diffeomorphisms due to Ruberman given in \cite{MR1671187}.
The relation between this invariant and PSC metrics will be given in \cref{subsection: Relation to the space of PSC metrics}.
We also note a generalization of this invariant of tuples of diffeomorphisms in \cref{subsection A generalized invariant} and interpret this generalized invariant as an obstruction to extensions of families of 4-manifolds in \cref{subsection: Invariant as an obstruction}.

\subsection{Definition of the invariant}
\label{subsection Definition of the invariant}

In this subsection
we define an invariant of $n$-tuples ($n \geq 1$) of commuting diffeomorphisms preserving a given $\spc$ structure on a $4$-manifold.
To do this, we will consider an $n$-parameter  family of \SW equations due to Nakamura~\cite{MR2015245, MR2644908}.
(For \cite{MR2015245}, there is a correction~\cite{MR2176601}.)
Strictly speaking, we will use a slight variant of the family: we shall describe the family as a subset of the space of perturbations while it has been given as an abstract fiber bundle on a torus in \cite{MR2015245, MR2644908}.
Our description is a higher-dimensional alanogue of Ruberman's $1$-parameter family ~\cite{MR1671187}.
(For a description of the family as a family on the torus, which is similar to Nakamura~\cite{MR2015245, MR2644908}, see \cref{remark on family on the torus}.)
To describe the notion of ``diffeomorphisms preserving a $\spc$ structure'' without using Riemannian metrics, we introduce a term {\it $spin^{c}_{GL}$ structure} as follows.
Fix a connected double covering $\tGL^{+}_{4}(\R)$ of $GL^{+}_{4}(\R)$, where $GL^{+}_{4}(\R)$ is the group of invertible real $4\times4$-matrices with $\det>0$.
Set 
\[
Spin^{c}_{GL}(4) := (\tGL^{+}_{4}(\R) \times U(1))/\pm1.
\]
We have the natural map $Spin^{c}_{GL}(4) \to GL_{4}^{+}(\R)$ as the map $Spin^{c}(4) \to SO(4)$.
For a given oriented $4$-manifold $X$, we denote by $\FrGL(X) \to X$ the frame bundle whose fiber at $x \in X$ is the set of oriented frames of $T_{x}X$.
We define a {\it $spin^{c}_{GL}$ structure} on $X$ as a $Spin^{c}_{GL}(4)$-bundle $P_{GL} \to X$ such that the $GL_{4}^{+}(\R)$-bundle $P_{GL} \times_{Spin^{c}_{GL}(4)} GL_{4}^{+}(\R) \to X$ coincides with $\FrGL(X)$.
In terms of classifying spaces, a $\spcGL$ structure is a lift $X \to BSpin^{c}_{GL}(4)$ of the classifying map $X \to BGL_{4}^{+}(\R)$ of $\FrGL(X)$ along the natural map $BSpin^{c}_{GL}(4) \to BGL_{4}^{+}(\R)$ induced from $Spin^{c}_{GL}(4) \to GL_{4}^{+}(\R)$:
\begin{align*}
\xymatrix{
     & BSpin^{c}_{GL}(4) \ar[d] \\
    X \ar[r] \ar[ru] &  BGL_{4}^{+}(\R).
    }
\end{align*}
An isomorphism class of $\spcGL$ structures corresponds to a homotopy class of such lifts.
Given a $\spcGL$ structure $\fraks$, a $\spc$ structure $\fraks_{g}$ is induced corresponding to each Riemannian metric $g$ on $X$.
For a fixed metric, an isomorphism class of $\spcGL$ structures corresponds one-to-one with an isomorphism class of $\spc$ structures.
We do not therefore distinguish $\spcGL$ structure from $\spc$ structure when we consider them at the level of isomorphism classes.

We here recall some basic facts related to wall-crossing in \SW theory.
Let $X$ be an oriented closed smooth $4$-manifold equipped with a homology orientation and $\fraks$ be a $\spcGL$ structure on $X$.
Here a homology orientation means an orientation of the vector space $H^1(X;\R) \oplus H^+(X;\R)$, where $H^+(X;\R)$ is a maximal positive-definite subspace of $H^2(X;\R)$ with respect to the intersection form of $X$.
We note that the determinant line bundle $L \to X$ is defined from $\fraks$ without using any metric on $X$.
We can therefore take a smooth reference connection $A_0$ of $L$ to be independent of the choice of metric.
For each Riemannian metric $g$ on $X$, we obtain the induced $\spc$ structure $\fraks_{g}$ and the spinor bundles $S^{\pm} = S^{\pm}_{g} \to X$ as usual.
Let us denote by $\Omega^+ = \Omega^+_g = \Gamma(\Lambda^{+}_{g})$ the space of self-dual $2$-forms on $X$ with respect to $g$.
A $U(1)$-connection $A$ on $L$ gives rise to the Dirac operator $D_A : \Gamma(S^+) \to \Gamma(S^-)$.
For a $U(1)$-connection $A$, a positive spinor $\Phi \in \Gamma(S^+)$, and an imaginary self-dual $2$-form $\mu \in i\Omega^+$, we call the equations
\begin{align*}
\begin{cases}
\rho(F^+_A + \mu) = \sigma(\Phi, \Phi),\\
D_A \Phi = 0
\end{cases}
\end{align*}
the (perturbed) \SW equations with respect to $(g, \mu)$.
Here $\rho : \Lambda^+ \to \mathfrak{su}(S^+)$ is the map obtained from the Clifford multiplication, $F^+_A$ is the self-dual part of the curvature $F_A$ of $A$, and $\sigma(\cdot, \cdot)$ is the quadratic form given by $\sigma(\Phi, \Phi) = \Phi \otimes \Phi^\ast - |\Phi|^2 \id/2$.
Let us denote by $\calM((g, \mu), \fraks)$ the moduli space of solutions to the perturbed \SW equations with respect to $(g, \mu)$ and $\fraks$.
The space of perturbations $\Pi = \Pi(X)$ is given by
\[
\Pi(X) := \Set{ (g, \mu) \in \Met(X) \times \Omega^2 | \mu \in \Omega^+_g } = \bigsqcup_{g \in \Met(X)}\Omega^+_g,
\]
where $\Met(X)$ is the space of Riemannian metrics on $X$.
This is a subbundle of the trivial bundle $\Met(X) \times \Omega^2 \to \Met(X)$.
Let us identify $\Met(X)$ with the zero-section of this subbundle.
For $A$ and $\Phi$, we call the equations
\begin{align*}
\begin{cases}
\rho(F^+_A) = \sigma(\Phi, \Phi),\\
D_A \Phi = 0
\end{cases}
\end{align*}
the unperturbed \SW equations with respect to $g$.
We call the subset of perturbations $\calW = \calW(X) \subset \Pi(X)$ defined by
\[
\calW(X) = \bigsqcup_{g \in \Met(X)}\calW_{g}(X), \quad
\calW_{g}(X) := F_{A_0}^{+_g} + \im{d^{+_g}} \subset \Omega^+_g
\]
the {\it wall}.
The wall $\calW(X)$ is obviously independent of the choice of $A_{0}$.
We define $\circPi = \circPi(X) = \bigsqcup_{g \in \Met(X)} \circPi_g(X)$ by
\[
\circPi_g(X) := \Pi_g(X) \setminus \calW_g(X).
 \]
The wall $\calW$ is of codimension-$b^+$ in $\Pi$ since $\Omega^+_g/\im(d^{+_g}) \simeq H^+(X;\R)$ for each $g$.
Since $\circPi$ is a fiber bundle whose fiber is homotopy equivalent to the sphere $S^{b^+ -1}$ and $\Met(X)$ is contractible, the total space $\circPi$ is also homotopy equivalent to $S^{b^+ -1}$.
Recall that, for $(g, \mu) \in \Pi$, the perturbed \SW equations with respect to $(g, \mu)$ admit a reducible solution if and only if $(g, \mu) \in \calW$.
Strictly speaking, as usual, we shall work on suitable Sobolev spaces.
However, we omit $L^2_{k}(\cdot)$ etc from our notation for simplicity.

Let us denote by $\Diff(X, \fraks)$ the group of diffeomorphisms preserving both the orientation of $X$ and $\fraks$.
Here we say that a diffeomorphism $f$ preserves $\fraks$ if $f$ satisfies $f^\ast \fraks \cong \fraks$ as $\spcGL$ structure.
Recall that, if $X$ is simply connected,  $f$ preserves $\fraks$ if and only if $f^\ast c_1(\fraks) = c_1(\fraks)$.
Set $n := -d(\fraks)$, where $d(\fraks)$ is the formal dimension of the (unparameterized) moduli space of solutions to the \SW equations:
$d(\fraks) = (c_1(\fraks)^2 - 2\chi(X) -3\sign(X))/4$.
Assume that
\[
n > 0 \quad {\rm and}\quad b^+(X) \geq n + 2.
\]
For commuting diffeomorphisms $f_1, \ldots, f_n \in \Diff(X, \fraks)$, 
we shall define
\[
\SWinv(f_1, \ldots, f_n; \fraks) \in \Z \ {\rm or}\ \Z/2.
\]
Here, if all of $f_1, \ldots, f_n$ preserve the given homology orientation, the number $\SWinv(f_1, \ldots, f_n; \fraks)$ is defined in $\Z$, and if at least one of $f_1, \ldots, f_n$ reverses the homology orientation, $\SWinv(f_1, \ldots, f_n; \fraks)$ is defined in $\Z/2$.

To define $\SWinv(f_1, \ldots, f_n; \fraks)$, we will use an $n$-parameter  family of \SW equations.
This is based on the idea of mapping torus due to Nakamura~\cite{MR2015245, MR2644908}.
(See \cref{remark on family on the torus}.)
For $k \in \{0,\ldots ,n\}$ and distinct indices $i_1, \ldots, i_k \in \{1, \ldots, n\}$ with $i_{1} < \cdots < i_{k}$, we shall construct a smooth generic map
\[
\vp_k(f_{i_1}, \ldots, f_{i_k}) : [0,1]^k \cong [0, 1]^{\{i_{1}, \ldots, i_{k}\}} \to \circPi
\]
inductively with respect to $k$.
Here, in the case that $k=0$, $\vp_0 = \vp_0(\emptyset)$ is a map from a point to $\circPi$.
By convention, $[0,1]^{0}$ is a single point denoted by $\{0\}$.
First note that all of $\pi_0(\circPi), \ldots, \pi_n(\circPi)$ are trivial since $b^+(X) \geq n+2$.
Let us take a generic point $p \in \circPi$.
This point can be regarded as a generic map $\vp_0 : \{0\} \to \circPi$.
Note that one can define the pull-back $f^\ast : \circPi \to \circPi$ for any $f \in \Diff(X, \fraks)$.
In the case that $k=1$, corresponding to each $i \in \{1, \ldots, n\}$, we can take a generic path from $p$ to $f_i^\ast p$ in $\circPi$ since $\pi_0(\circPi)$ is trivial.
This is given by a generic map $\vp_1(f_i) : [0,1] \to \circPi$.
In the case that $k \geq 2$, for distinct $i$ and $j$ with $i < j$, the pulled-back path $f_i^\ast \vp_1(f_j) : [0,1] \to \circPi$ gives a path from $f_i^\ast p$ to $f_i^\ast f_j^\ast p$, and $f_j^\ast \vp_1(f_i) : [0,1] \to \circPi$ gives a path from $f_j^\ast p$ to $f_i^\ast f_j^\ast p$
 since $f_i f_j = f_j f_i$.
Because $\pi_1(\circPi)$ is trivial, we can take a generic map $\vp_2(f_i, f_j) : [0,1]^2 \to \circPi$ such that  the map $\vp_2(f_i, f_j)$ coincides with
\begin{itemize}
\item $\vp_1(f_i)$ on $[0,1] \times \{0\}$,
\item $\vp_1(f_j)$ on $\{0\} \times [0,1]$,
\item $f_j^\ast \vp_1(f_i)$ on $[0,1] \times \{1\}$, and
\item $f_i^\ast \vp_1(f_j)$ on $\{1\} \times [0,1]$.
\end{itemize}
\begin{figure}
\begin{center}
\scalebox{1}{
\begin{tikzpicture}
[xscale = 1.2, yscale = 1.2]
\draw [thick] (4,2) -- (6,2);
\draw [thick] (4,0) -- (6,0);
\draw [thick] (4,2) -- (4,0);
\draw [thick] (6,2) -- (6,0);
\fill(4,2) circle (2.2 pt);
\fill(6,2) circle (2.2 pt);
\fill(4,0) circle (2.2 pt);
\fill(6,0) circle (2.2 pt);
\draw(3.8,1-1.4) node {$p$};
\draw(6.3,1-1.4) node {$f_{i}^{\ast}p$};
\draw(3.7,2.4) node {$f_{j}^{\ast}p$};
\draw(7,2.4) node {$f_{i}^{\ast}f_{j}^{\ast}p = f_{j}^{\ast}f_{i}^{\ast}p$};
\draw(5,1-1.4) node {$\vp_{1}(f_{i})$};
\draw(5,2.4) node {$f_{j}^{\ast}\vp_{1}(f_{i})$};
\draw(3.5,1) node {$\vp_{1}(f_{j})$};
\draw(6.7,1) node {$f_{i}^{\ast}\vp_{1}(f_{j})$};
\draw(5,1) node {$\vp_{2}(f_{i},f_{j})$};
\fill [gray, opacity=.3] (4,2) -- (6, 2) -- (6, 0) -- (4,0);
\end{tikzpicture}
}
\end{center}
\caption{$2$-parameter family corresponding to $\vp_{2}(f_{i},f_{j})$}
\label{figure : k2}
\end{figure}
(See \cref{figure : k2}.)
We can easily extend this construction to higher-dimensional families as follows.
For $k \in \{1, \ldots, n-1\}$, assume that generic maps $\vp_{l}(\cdot, \ldots, \cdot)$ $(0 \leq l \leq k)$ are already given.
Since $\pi_{k}(\circPi)$ is trivial, for distinct indices $i_1, \ldots, i_{k+1}$ with $i_{1} < \cdots < i_{k+1}$, we can take a generic map
\[
\vp_{k+1}(f_{i_1}, \ldots, f_{i_{k+1}}) : [0,1]^{k+1} \to \circPi
\]
satisfying the following condition.
Note that a codimension-$1$ face of $[0,1]^{k+1}$ is one of
\begin{align}
F_{0,j}^{k} := \{(\ast, \ldots, \ast, 0, \ast, \ldots, \ast)\}
\label{eq: facet time-0}
\end{align}
and 
\begin{align}
F_{1,j}^{k} := \{(\ast, \ldots, \ast, 1, \ast, \ldots, \ast)\}
\label{eq: facet time-1}
\end{align}
for some $j \in \{1\, \ldots , k+1\}$, where $0$ and $1$ are in the $j$-th coordinates.
We require that the map $\vp_{k+1}(f_{i_1}, \ldots, f_{i_{k+1}})$ coincides with
\begin{itemize}
\item $\vp_{k}(f_{i_1}, \ldots, \hat{f}_{i_j}, \ldots,  f_{i_{k+1}})$ on $F_{0,j}^{k} \cong [0,1]^{k}$, and
\item ${f}_{i_j}^\ast \vp_{k}(f_{i_1}, \ldots, \hat{f}_{i_j}, \ldots,  f_{i_{k+1}})$ on $F_{1,j}^{k} \cong [0,1]^{k}$,
\end{itemize}
where the notation $\hat{f}_{i_j}$ means that the component is removed.

Though the above procedure, we obtain a generic $n$-parameter family
\[
\vp_n(f_{1}, \ldots, f_{n}) : [0,1]^n \to \circPi
\]
and we can consider the family of (perturbed) \SW equations parameterized on $[0,1]^n$ by $\vp_n(f_{1}, \ldots, f_{n})$.
Since $d(\fraks) = -n$, the moduli space of solutions to the \SW equations vanishes on each codimension $\geq 1$ face of this $n$-parameter family.
Let us denote by $\calM(\vp_n(f_{1}, \ldots, f_{n}), \fraks)$ the parameterized moduli space on $[0,1]^n$ by $\vp_n(f_{1}, \ldots, f_{n})$, namely, 
\begin{align}
\calM(\vp_n(f_{1}, \ldots, f_{n}), \fraks) := \bigsqcup_{t \in [0,1]^n} \calM(\vp_n(f_{1}, \ldots, f_{n})(t), \fraks).
\label{fundamental moduli on Euclidean space}
\end{align}
By the compactness of the usual (i.e. unparameterized) moduli space of solutions to the \SW equations and that of the parameter space $[0,1]^n$, the parameterized moduli space $\calM(\vp_n(f_{1}, \ldots, f_{n}), \fraks)$ is also compact.
Since we have fixed a homology orientation, $\calM(\vp_n(f_{1}, \ldots, f_{n}), \fraks)$ is oriented.
We can therefore define the integer 
\[
\SWinv(f_1, \ldots, f_n; \fraks; \vp_{\bullet}) := \# \calM(\vp_n(f_{1}, \ldots, f_{n}), \fraks) \in \Z
\]
by counting with signs the points of the parameterized moduli space, which is a $0$-dimensional compact oriented manifold.

\begin{rem}
\label{remark on family on the torus}
We remark that the definition of $\SWinv(f_1, \ldots, f_n; \fraks; \vp_{\bullet})$ above can be interpreted as a counting for a family on the $n$-torus.
We note that, for each $j \in \{1, \ldots, n\}$, $f_{j}$ induces a diffeomorphism
\begin{align}
\calM(\vp_{n}(f_{1}, \ldots, f_{n}), \fraks)|_{F_{0,j}^{n-1}} \cong \calM(\vp_{n}(f_{1}, \ldots, f_{n}), \fraks)|_{F_{1,j}^{n-1}}.
\label{fundamental diffeo between moduli}
\end{align}
As remarked in Ruberman~\cite{MR1671187}, 
the diffeomorphism \eqref{fundamental diffeo between moduli} between moduli spaces is independent of the choice of lift of $f_j$ to an isomorphism at the level of $\spc$ structures, since the ambiguity of the choice is absorbed into the gauge group.
By identifying
\[
\bigsqcup_{t \in F_{0,j}^{n-1}} \calM(\vp_n(f_{1}, \ldots, f_{n})(t), \fraks)\ {\rm and}\ \bigsqcup_{t \in F_{1,j}^{n-1}} \calM(\vp_n(f_{1}, \ldots, f_{n})(t), \fraks)
\]
via the diffeomorphism \eqref{fundamental diffeo between moduli} in \eqref{fundamental moduli on Euclidean space}, we obtain a parameterized moduli space on $T^{n}$.

\end{rem}

In the argument above, of course, it is sufficient to assume $b^+(X) \geq n + 1$ to avoid the wall.
Namely, we have not used $\pi_{n}(\circPi) = 0$.
We use the assumption $b^+(X) \geq n+2$ to do the argument by cobordism in the following lemma.
Using the idea of Theorem 2.2 in Ruberman~\cite{MR1671187}, we have:

\begin{lem}
\label{lem : well-def. of simple invariant}
Assume that  $f_1, \ldots, f_n \in \Diff(X, \fraks)$ are commuting.
\begin{enumerate}
\item If all of $f_1, \ldots, f_n$ preserve the given homology orientation, then the element
\[
\SWinv(f_1, \ldots, f_n; \fraks; \vp_{\bullet})
\]
 in $\Z$ is independent of the choice of $\vp_{\bullet}$.
\item If at least one of $f_1, \ldots, f_n$ reverses the homology orientation, then the element
\[
\SWinv(f_1, \ldots, f_n; \fraks; \vp_{\bullet}) \mod 2
\]
in $\Z/2$ is independent of the choice of $\vp_{\bullet}$.
\end{enumerate}
\end{lem}

\begin{proof}
For two families $\vp_{\bullet}$ and $\vp_{\bullet}'$, we can take a generic map
\[
\psi : [0,1]^{n+1} = [0,1]^n \times [0,1] \to \circPi
\]
such that
\begin{itemize}
\item on $[0,1]^n \times \{0\}$, the map $\psi$ coincides with $\vp_n(f_{1}, \ldots, f_{n})$,
\item on $[0,1]^n \times \{1\}$, the map $\psi$ coincides with $\vp_n'(f_{1}, \ldots, f_{n})$, and
\item for each $j \in \{1, \ldots, n\}$, the map $\psi$ satisfies $\psi|_{F_{1,j}^{n-1} \times [0,1]} = f_j^{\ast}\psi|_{F_{0,j}^{n-1} \times [0,1]}$
\end{itemize}
since $\pi_{n}(\circPi)$ is trivial.
Here $F_{0,j}^{n-1}$ and $F_{1,j}^{n-1}$ are the facets of $[0,1]^n$ obtained by putting $k=n-1$ in \eqref{eq: facet time-0} and \eqref{eq: facet time-1}.  
Since there are no reducibles on the parameterized moduli space given by $\psi$, we have a $1$-dimensional compact manifold with boundary as a parameterized moduli space on $[0,1]^{n+1}$ given by $\psi$.
Its boundary components are $\calM(\vp_n(f_{1}, \ldots, f_{n}), \fraks),\ \calM(\vp_n'(f_{1}, \ldots, f_{n}), \fraks)$, and the moduli spaces parameterized on $F_{0,j}^{n-1} \times [0,1]$ and $F_{1,j}^{n-1} \times [0,1]$ for each $j \in \{1, \ldots, n\}$.
Note that $f_{j}$ induces a diffeomorphism similar to \eqref{fundamental diffeo between moduli} between the last two components:
\begin{align}
\calM(\psi, \fraks)|_{F_{0,j}^{n-1} \times [0,1]} \cong \calM(\psi, \fraks)|_{F_{1,j}^{n-1} \times [0,1]}.
\label{diffeomorphim on moduli spcaces}
\end{align}
If $f_j$ preserves the homology orientation, this diffeomorphism \eqref{diffeomorphim on moduli spcaces} preserves the orientation of the moduli spaces.
Hence, if all of $f_1, \ldots, f_n$ preserve the homology orientation, all contributions in the counting argument on $F_{0,j}^{n-1} \times [0,1]$ and $F_{1,j}^{n-1} \times [0,1]$ are canceled in $\Z$.
We thus have $\#\calM(\vp_n(f_{1}, \ldots, f_{n}), \fraks) = \# \calM(\vp_n'(f_{1}, \ldots, f_{n}), \fraks)$ in $\Z$.
If there exists $j$ such that $f_j$ reverses the homology orientation, this cancellation holds over $\Z/2$.
\end{proof}

By \cref{lem : well-def. of simple invariant}, for commuting diffeomorphisms $f_1, \ldots, f_n \in \Diff(X, \fraks)$, we can define
\[
\SWinv(f_1, \ldots, f_n; \fraks) \in \Z \ {\rm or}\ \Z/2
\]
as $\SWinv(f_1, \ldots, f_n; \fraks; \vp_{\bullet})$ for a family of perturbations $\vp_{\bullet}$.
Here, if all of $f_1, \ldots, f_n$ preserve the given homology orientation, then $\SWinv(f_1, \ldots, f_n; \fraks)$ is defined in $\Z$, and if at least one of $f_1, \ldots, f_n$ reverses the homology orientation, then $\SWinv(f_1, \ldots, f_n; \fraks)$ is defined in $\Z/2$.
This is a higher-dimensional generalization of Ruberman's invariant of diffeomorphisms given in~\cite{MR1671187}.
Namely, the case that $n=1$ for the above $\SWinv(f_1, \ldots, f_n; \fraks)$ is Ruberman's invariant.

\begin{rem}
By a similar procedure to define $\SWinv(f_1, \ldots, f_n; \fraks)$, one can also define a higher-dimensional analogue of Ruberman's refined invariant given in \cite{MR1874146} written as $\SWinv_{tot}$.
However, at this stage, the author cannot find any application of the higher-dimensional $\SWinv_{tot}$ which cannot obtained from either the original $\SWinv_{tot}$ or $\SWinv(f_1, \ldots, f_n; \fraks)$ defind above.
\end{rem}

\subsection{Relation to the space of PSC metrics}
\label{subsection: Relation to the space of PSC metrics}

In this subsection we follow the all of the notations of \cref{subsection Definition of the invariant}.
Let us denote by $\PSC(X)$ the space of PSC metrics on $X$.
As Ruberman's invariant~\cite{MR1874146}, the topology of $\PSC(X)$ relates to the triviality of this invariant.

\begin{prop}
\label{vanishing of homotopy groups and invariant}
Suppose that $\PSC(X)$ is non-empty and that $\pi_i(\PSC(X))$ is trivial for all $i \in \{0, \ldots ,n-1\}$.
Then, for any commuting elements $f_1, \ldots, f_n \in \Diff(X, \fraks)$,
\[
\SWinv(f_1, \ldots, f_n; \fraks) = 0
\]
holds.
\end{prop}

\begin{proof}
Let us regard $\PSC(X) \subset \Met(X) \subset \Pi$.
We first consider the case that either $c_1(\fraks)^2 > 0$, or $c_1(\fraks)^2 = 0$ and $c_1(\fraks)$ is not torsion.
In these cases, we have $\Met(X) \cap \calW = \emptyset$, in particular $\PSC(X) \cap \calW = \emptyset$.
By our assumption, we can construct maps
\[
\vp_k(f_{i_1}, \ldots, f_{i_k}) : [0,1]^k \to \PSC(X)
\]
for distinct $i_1, \ldots, i_k \in \{1, \ldots, n\}$ inductively with respect to $k \in \{0, \ldots ,n\}$ as in the definition of $\SWinv(f_1, \ldots, f_n; \fraks)$.
Note that there is no reducible in the moduli space $\calM(\vp_n(f_{1}, \ldots, f_{n}), \fraks)$.
We therefore have $\calM(\vp_n(f_{1}, \ldots, f_{n}), \fraks) = \emptyset$ by the a priori estimate for spinors of solutions to the \SW equations.
We thus obtain $\SWinv(f_1, \ldots, f_n; \fraks) = 0$.

Next let us consider the case that $c_1(\fraks)^2 < 0$.
In this case, $\Met(X) \cap \calW$ is a codimension-$b^+$ subspace of $\Met(X)$.
(See Subsubsection~4.3.3 in Donaldson--Kronheimer~\cite{MR1079726}.)
Since $\PSC(X)$ is an open subspace of $\Met(X)$, $\PSC(X) \cap \calW$ is a codimension-$b^+$ subspace of $\PSC(X)$.
All of $\pi_0(\PSC(X) \setminus \calW), \ldots, \pi_{n-1}(\PSC(X) \setminus \calW)$ are hence trivial by our assumption.
We can therefore construct maps into $\PSC(X) \setminus \calW$
\[
\vp_k(f_{i_1}, \ldots, f_{i_k}) : [0,1]^k \to \PSC(X) \setminus \calW
\]
inductively.
Since we avoid the wall, we obtain $\SWinv(f_1, \ldots, f_n; \fraks) = 0$ by the same argument above.

Finally, let us consider the case that $c_1(\fraks)$ is torsion.
In this case, we have $\PSC(X) \subset \Met(X) \subset \calW$.
 We have to therefore avoid the wall by ``lifting" the above argument to $\circPi$.
We can construct maps
\[
\bar{\vp}_k(f_{i_1}, \ldots, f_{i_k}) : [0,1]^k \to \PSC(X)
\]
inductively by the same way, and next we construct a lift $\vp_k(\cdot, \ldots, \cdot)$ of $\bar{\vp}_k(\cdot, \ldots, \cdot)$ to $\circPi$ inductively as follows.
Set $K = \bar{\vp}_n(f_1, \ldots, f_n)([0,1]^n)$.
Let $D = \bigsqcup_{g} D_g \to K$ be the disk bundle with a small radius of $\bigsqcup_{g} \Omega^+_g \to K$.
(The precise condition on the radius of the disk which we have to assume is given below.)
Then, 
\[
\circPi \cap D \to K
\]
is a fiber bundle whose fiber is homotopy equivalent to $S^{b^+-1}$.
Let $E \to [0,1]^n$ denote the pull-back of $\circPi \cap D \to K$ by $\bar{\vp}_n(f_1, \ldots, f_n) : [0,1]^n \to K (\subset \PSC(X))$.
Fix a small element of $\circPi_{\bar{\vp}_0(0)} \cap \Omega^+_{\bar{\vp}_0(0)}$.
This gives a lift
\[
\vp_0 : \{0\} \to \circPi \cap D
\]
of $\bar{\vp}_0 : \{0\} \to K$.
From this element and the pulled-back elements by $f_i$ $(i = 1, \ldots, n)$, we have  a section of the bundle $E$ over the $0$-dimensional faces of $[0,1]^n$.
Since the fiber of $E$ is homotopy equivalent to $S^{b^+-1}$, this section can be extended to a section on the $1$-dimensional faces of $[0,1]^n$ so that the extended section gives a lift
\[
\vp_1(f_i) : [0,1] \to \circPi \cap D 
\]
of $\bar{\vp}_1(f_i) : [0,1] \to K$ for $i = 1, \ldots, n$.
Next, for $k \in \{1, \ldots, n-1\}$, assume that we have fixed lifts $\vp_{l}(\cdot, \ldots, \cdot)$ of $\bar{\vp}_{l}(\cdot, \ldots, \cdot)$  $(0 \leq l \leq k)$.
Then we can extend these lifts on the boundary $\del [0,1]^{k+1}$ to ones on all of $[0,1]^{k+1}$ by the same way.
By this inductive construction we can obtain a lift
\[
\vp_{k}(f_{i_1}, \ldots, f_{i_k}) : [0,1]^k \to \circPi \cap D 
\]
of $\bar{\vp}_{k}(f_{i_1}, \ldots, f_{i_k})$ for any $k \in \{0, \ldots, n\}$ and any distinct $i_1, \ldots, i_k$.
We now give the precise condition on the radius of the above disk bundle $D \to K$.
For any $g \in \PSC(X)$, let us denote by $s_g : X \to \R$ the scalar curvature with respect to $g$.
Then the subspace
\begin{align}
\Set{ \mu \in \Omega^+_g | -\min_{x \in X} s_g(x) + 2|\mu| < 0,\ -\min_{x \in X} s_{f_i^\ast g}(x) + 2|f_i^\ast \mu| < 0\ (1 \leq i \leq n)}
\label{open subspace of Omegaplus}
\end{align}
of $(\Omega^+_g, |\cdot|)$ is an open neighborhood of the origin of $\Omega^+_g$, where $|\cdot|$ denotes the $C^0$-norm.
Let $\epsilon_g > 0$ be the maximal radius of open balls centered at the origin included in this open neighborhood \eqref{open subspace of Omegaplus} of the origin.
Then $\epsilon := \min_{g \in K} \epsilon_g$ also satisfies $\epsilon > 0$ since $\epsilon_g$ continuously depends on $g$.
Use this $\epsilon$ as the radius of the disk bundle.
Then, by the a priori estimate for spinors, we have $\calM(\vp_k(f_{1}, \ldots, f_{n}), \fraks) = \emptyset$,
and hence $\SWinv(f_1, \ldots, f_n; \fraks) = 0$.
\end{proof}

\subsection{A generalized invariant}
\label{subsection A generalized invariant}

In this subsection we note a generalization of the invariant $\SWinv(f_1, \ldots, f_n; \fraks)$ based on the idea of the construction of a family on the $n$-torus in \cref{remark on family on the torus}.
Although the contents of this subsection and \cref{subsection: Invariant as an obstruction} are not necessary for the proof of \cref{application to PSC}, we will use them in \cref{subsection: Application to the extension problem for families of 4manifolds}, where we discuss an obstruction to extensions of families of $4$-manifolds.
We note that, if we consider a family of $\spc$ $4$-manifolds, the generalized invariant is a special case of the usual family \SW invariant given in Li--Liu~\cite{MR1868921}.
However, since our structure group of a family is $\Diff(X, \fraks)$, we need the following discussion.
(See the following \cref{remark on Diff rather than Aut}.)

Let $(X, \fraks)$ be a closed smooth $\spc$ $4$-manifold equipped with a homology orientation $\mathcal{O}$, $M$ be an $n$-dimensional closed smooth manifold.
Suppose that $d(\fraks) = -n$ and $b^{+}(X) \geq n + 2$.
Let $\rho : M \to B\Diff(X, \fraks)$ be a continuous map and $E_{X} \to M$ be the bundle corresponding to $\rho$.
This is a bundle whose fiber is $X$ with structure group $\Diff(X, \fraks)$.
Suppose that $E_{X}$ satisfies the following condition on smoothness:
we can choose transition functions $\{ g_{\alpha\beta} : U_{\alpha} \cap U_{\beta} \to \Diff(X, \fraks) \}_{\alpha, \beta}$ of $E_{X}$ so that the map $(U_{\alpha} \cap U_{\beta}) \times X \to X$ given by $(p,x) \mapsto g_{\alpha\beta}(p)x$ is smooth for any $\alpha, \beta$.
Let $\Diff^{+}(X)$ be the group of diffeomorphism preserving the orientation of $X$.
Via the inclusion $\Diff(X, \fraks) \inc \Diff^{+}(X)$,  the actions of $\Diff^{+}(X)$ on $X$ and on $\circPi(X)$ induce bundles $E_{X} \to M$ and $E_{\circPi} \to M$ whose fibers are $X$ and $\circPi(X)$ respectively.
Let us take a smooth section $s : M \to E_{\circPi}$.
As remarked in Section~5 in Nakamura~\cite{MR2644908}, we can consider the parameterized moduli space $\calM(\rho, s, \fraks)$ as follows.

\begin{rem}
\label{remark on Diff rather than Aut}
Before giving the construction of $\calM(\rho, s, \fraks)$, we note that, a priori, it is non-trivial how to construct such a moduli space $\calM(\rho, s, \fraks)$ using a section of $E_{\circPi}$.
To explain it, assume that we have a homomorphism $\tilde{\rho} : M \to B\Aut(X, \fraks)$, where $\Aut(X, \fraks)$ is the group of pairs $(f, \tilde{f})$ consisting of $f \in \Diff(X,\fraks)$ and a $Spin_{GL}^{c}(4)$-equivariant map $\tilde{f} : P_{GL} \to P_{GL}$ satisfying
\[
\xymatrix{
    P_{GL} \ar[d] \ar[r]^{\tilde{f}} & P_{GL} \ar[d] \\
    X \ar[r]_{f} & X  \ar@{}[lu]|{\circlearrowright},
    }
\quad
\xymatrix{
    P_{GL} \ar[d] \ar[r]^{\tilde{f}} & P_{GL} \ar[d] \\
    \FrGL(X) \ar[r]_{df} & \FrGL(X)  \ar@{}[lu]|{\circlearrowright}.
    }
\]
We have a natural surjection $\Aut(X, \fraks) \to \Diff(X, \fraks)$, and therefore obtain 
\[
X \to \tilde{E}_{X} \to M,\ \circPi(X) \to \tilde{E}_{\circPi} \to M
\]
using $\tilde{\rho}$.
Each fiber of $\tilde{E}_{X} \to M$ has a natural $\spcGL$ $4$-manifold structure, and therefore a $\spc$ structure if we give a metric.
From this, if a section $s : M \to \tilde{E}_{\circPi}$ is given, we can consider a family of \SW equations with respect to $s$, and hence obtain the parameterized moduli space $\calM(\tilde{\rho}, s, \fraks)$ as usual.
However, each fiber of the bundle $E_{X} \to M$ above has {\it no} natural $\spcGL$ structure.
We cannot therefore obtain the parameterized moduli space from a section of $E_{\circPi} \to M$ by the entirely same way to construct $\calM(\tilde{\rho}, s, \fraks)$.
To define $\calM(\rho, s, \fraks)$ above, we need a ``local version'' of the argument used to show that the diffeomorphism \eqref{fundamental diffeo between moduli} is independent of the choice of lift of $f_{j}$.
\end{rem}

Let us take an open covering $\{U_{\alpha}\}_{\alpha}$ of $M$ satisfying that $U_{\alpha} \cap U_{\beta}$ is contractible for any $\alpha, \beta$.
Henceforth we abbreviate $U_{\alpha} \cap U_{\beta}$ to $U_{\alpha\beta}$ and $U_{\alpha} \cap U_{\beta} \cap U_{\gamma}$ to $U_{\alpha\beta\gamma}$ respectively.
Take a system of local trivializations of $E_{X} \to M$ on this covering
and let $\{ g_{\alpha\beta} : U_{\alpha\beta} \to \Diff(X, \fraks) \}_{\alpha, \beta}$ be the transition functions corresponding to this system of local trivializations.
Since $U_{\alpha\beta}$ is contractible for each $\alpha, \beta$, there exists a lift $\tilde{g}_{\alpha\beta} : U_{\alpha\beta} \to \Aut(X, \fraks)$ of $g_{\alpha\beta}$.
Let $\Aut_{X}(\fraks)$ be the kernel of $\Aut(X, \fraks) \to \Diff(X, \fraks)$:
we have the exact sequence
\[
1 \to \Aut_{X}(\fraks) \to \Aut(X, \fraks) \to \Diff(X, \fraks) \to 1.
\]
The group $\Aut_{X}(\fraks)$ is isomorphic to the gauge group $\gauge \cong \Map(X, S^{1})$.
Note that we have $\tilde{g}_{\alpha\beta}\tilde{g}_{\beta\gamma}\tilde{g}_{\gamma\alpha}(p) \in \Aut_{X}(\fraks) \cong \gauge$ for any $p \in U_{\alpha\beta\gamma}$ since $\{g_{\alpha\beta}\}$ satisfies the cocycle condition.
The given section $s : M \to E_{\circPi}$ corresponds to a system of maps $\{s_{\alpha} : U_{\alpha} \to \circPi(X)\}_{\alpha}$ satisfying that $s_{\alpha} = g_{\alpha\beta} \cdot s_{\beta}$ on $U_{\alpha\beta}$.
Here the action $g_{\alpha\beta} \cdot s_{\beta}$ is given by the action of $\Diff(X, \fraks)$ on $\circPi(X)$ via $\Diff(X, \fraks) \inc \Diff^{+}(X)$, namely, $g_{\alpha\beta} \cdot s_{\beta} = g_{\alpha\beta}^{\ast} s_{\beta}$.
For each $\alpha$, let us write 
\[
\calM(\rho, s_{\alpha}, \fraks) := \bigsqcup_{p \in U_{\alpha}} \calM(s_{\alpha}(p), \fraks),
\]
where $\calM(s_{\alpha}(p), \fraks)$ is the moduli space with respect to $s_{\alpha}(p) \in \circPi(X)$ in usual sense, and
$s_{\alpha}$ is regarded as a section of the trivial bundle.
For each point $p \in U_{\alpha\beta}$,
we obtain an invertible map
\[
\tilde{g}_{\alpha\beta}(p)^{\ast} : \calM(s_{\beta}(p), \fraks) \to \calM(g_{\alpha\beta}(p)^{\ast} s_{\beta}(p), \fraks).
\]
Since the relation $s_{\alpha} = g_{\alpha\beta} \cdot s_{\beta} = g_{\alpha\beta}^{\ast} s_{\beta}$ on $U_{\alpha\beta}$ holds, we eventually have
\[
\tilde{g}_{\alpha\beta}^{\ast} : \calM(\rho, s_{\beta}, \fraks)|_{U_{\alpha\beta}} \to \calM(\rho, s_{\alpha}, \fraks)|_{U_{\alpha\beta}}.
\]
The composition
\[
\tilde{g}_{\alpha\beta}^{\ast} \circ \tilde{g}_{\beta\gamma}^{\ast} \circ \tilde{g}_{\gamma\alpha}^{\ast} : \calM(\rho, s_{\alpha}, \fraks) |_{U_{\alpha\beta\gamma}} \to \calM(\rho, s_{\alpha}, \fraks)|_{U_{\alpha\beta\gamma}}
\]
coincides with the identity since $\tilde{g}_{\alpha\beta}\tilde{g}_{\beta\gamma}\tilde{g}_{\gamma\alpha}(p) \in \Aut_{X}(\fraks) \cong \gauge$ holds.
This is again a consequence of the definition of the moduli space: it is the quotient space by the gauge group.
We can therefore obtain the well-defined quotient space
\[
\calM(\rho, s, \fraks) := \bigsqcup_{\alpha} \calM(\rho, s_{\alpha}, \fraks) / \sim,
\]
where the equivalence relation $\sim$ is given by the invertible maps $\{\tilde{g}^{\ast}_{\alpha\beta}\}$.
If $s$ is generic, each $s_{\alpha}$ is also generic and $\tilde{g}_{\alpha\beta}^{\ast}$ is a diffeomorphism between smooth manifolds.
The moduli space $\calM(\rho, s, \fraks)$ is hence also a smooth manifold.
If we have a $1$-parameter family of sections $\{s_{t} : M \to E_{\circPi}\}_{t \in [0,1]}$, it gives rise to a parameterized moduli space $\bigsqcup_{t \in [0,1]} \calM(\rho, s_{t}, \fraks)$ by the same way.
We can therefore do an argument by cobordism, and hence can define the invariant 
\begin{align}
\SWinv(\rho; \fraks) \in \Z \text{ or } \Z/2
\label{generalized invariant result}
\end{align}
as follows.
Since $b^{+}(X) \geq n+2$, for given two generic sections $s_{0}, s_{1} : M \to E_{\circPi}$, we can take a path of sections $\{s_{t} : M \to E_{\circPi}\}_{t \in [0,1]}$ between $s_{0}$ and $s_{1}$ such that $s_{\bullet}$ is generic as a map from $M \times [0,1]$ to $E_{\circPi}$.
We therefore obtain the parameterized moduli space $\bigsqcup_{t \in [0,1]}\calM(\rho, s_{t}, \fraks)$, whose boundary components are $\calM(\rho, s_{0}, \fraks)$ and $\calM(\rho, s_{1}, \fraks)$.
Let $\Diff(X, \fraks, \mathcal{O})$ be the group defined by
\[
\Diff(X, \fraks, \mathcal{O}) = \Set{f \in \Diff(X, \fraks) | f^{\ast}\mathcal{O} = \mathcal{O}}.
\]
By the argument by cobordism above, if $\rho(M) \subset B\Diff(X, \fraks, \mathcal{O})$ holds and if $M$ is oriented, the integer
\[
\SWinv(\rho; \fraks; s) := \# \calM(\rho, s, \fraks) \in \Z
\]
is independent of the choice of $s$, and otherwise $\SWinv(\rho; \fraks; s) \mod 2$ in $\Z/2$ is independent of $s$.
We can therefore define the invariant \eqref{generalized invariant result} as $\SWinv(\rho; \fraks; s)$ for a generic section $s$.

Given a group homomorphism $\Phi : \pi_{1}(M) \to \Diff(X, \fraks)$, we define 
$\SWinv(\Phi, \fraks) \in \Z \text{ or } \Z/2$ as $\SWinv(\rho, \fraks)$ for the classifying map $\rho$ of the bundle $X \to \tilde{M} \times_{\pi_{1}(M)} \to M$,
where $\tilde{M}$ is the universal covering of $M$.

\begin{ex}
\label{interpretation of SW for tuples}

Let $f_{1}, \ldots, f_{n} \in \Diff(X, \fraks)$ be commuting elements and $\Phi$ denote the map $\Z^{n} \to \left< f_{1}, \ldots, f_{n} \right> \subset \Diff(X, \fraks)$ given as $(k_{1}, \ldots, k_{n}) \mapsto f_{1}^{k_{1}} \cdots f_{n}^{k_{n}}$.
As described in \cref{remark on family on the torus}, the invariant $\SWinv(f_1, \ldots, f_n; \fraks)$ defined in \cref{subsection Definition of the invariant} can be  reinterpreted as the counting of the points of a parameterized moduli space on $T^{n}$.
By the construction of $\SWinv(f_1, \ldots, f_n; \fraks)$, the invariant $\SWinv(\Phi; \fraks)$ coincides with $\SWinv(f_1, \ldots, f_n; \fraks)$.
\end{ex}

\begin{rem}
\label{remark on use of obstruction theory}

Via the interpretation of $\SWinv(f_1, \ldots, f_n; \fraks)$ in \cref{interpretation of SW for tuples}, we can prove \cref{vanishing of homotopy groups and invariant} using obstruction theory.
We first consider the case that either $c_1(\fraks)^2 > 0$, or $c_1(\fraks)^2 = 0$ and $c_1(\fraks)$ is not torsion.
Via the composition
\[
\Phi : \Z^{n} \to \left< f_{1}, \ldots, f_{n} \right> \subset \Diff(X, \fraks) \inc \Diff^{+}(X),
\]
the action of $\Diff^{+}(X)$ on $\PSC(X)$ gives rise to a bundle 
$E_{\PSC} \to T^{n}$ whose fiber is $\PSC(X)$.
Note that $E_{\PSC} \subset E_{\circPi}$ holds.
The obstructions for the existence of a section of the bundle $E_{\PSC}$ live in $H^{i+1}(T^{n}; \hat{\pi}_{i}(\PSC(X)))$ $(i \in \{0, \ldots, n-1\})$, where $\hat{\pi}_{i}(\PSC(X))$ is a local system whose fiber is $\pi_{i}(\PSC(X))$.
Hence there exists a section $s : T^{n} \to E_{\PSC}$ if $\pi_i(\PSC(X))$ is trivial for any $i \in \{0, \ldots ,n-1\}$.
Using this section $s$ to calculate $\SWinv(f_{1}, \ldots, f_{n};\fraks) = \SWinv(\Phi ; \fraks)$, we have $\SWinv(f_{1}, \ldots, f_{n};\fraks) = 0$.
In the cases that $c_{1}(\fraks)^{2} < 0$ and that $c_{1}(\fraks)$ is torsion, by replacing $E_{\PSC}$ with the bundles given by the action of $\Diff^{+}(X)$ whose fibers are $\PSC(X) \cap \calW$ and $\circPi \cap D$ given in the proof of \cref{vanishing of homotopy groups and invariant} respectively, 
we can similarly show that $\SWinv(f_{1}, \ldots, f_{n};\fraks) = 0$.
\end{rem}

The argument of \cref{remark on use of obstruction theory} immediately gives the generalization of \cref{vanishing of homotopy groups and invariant}:

\begin{prop}
Let $(X, \fraks)$ be a closed smooth ${\it spin}^{c}$ $4$-manifold with $d(\fraks) = -n$ $(n \geq 1)$ and $M$ be an smooth closed $n$-dimensional manifold.
Suppose that $\PSC(X) \neq \emptyset$ holds and that $\pi_i(\PSC(X))$ is trivial for any $i \in \{0, \ldots ,n-1\}$.
Then,  for any group homomorphism $\Phi : \pi_{1}(M) \to \Diff(X, \fraks)$,
$\SWinv(\Phi; \fraks) = 0$ holds.
\end{prop}

\begin{rem}
In a subsequent paper \cite{Konno3}, the author developed family gauge theory over arbitrary topological space.
In \cite{Konno3}, the author showed that the smoothness of the base space of a family of $4$-manifold is not necessary to discuss family gauge theory.
More precisely, if the base space is not smooth, of course we cannot hope that the parameterized moduli space is smooth, but still we can ``count'' the parametrized moduli space if we can define the fundamental class of the base space (e.g. the case that the base space is a closed topological manifold, or more generally homology manifold).
The main reason why we can consider a family over a non-smooth base space is the author used the so-called virtual neighborhood technique in \cite{Konno3}.
The argument given in \cite{Konno3} ensures that any subtle stuff relating to the base space (e.g. non-reduced structure) does not cause any problem in our argument.
\end{rem}

\subsection{Obstruction to extension of families of $4$-manifolds}
\label{subsection: Invariant as an obstruction}

We can use our generalized invariant $\SWinv(\Phi;\fraks)$ to give an obstruction to extensions of families of $4$-manifolds with structure group $\Diff(X, \fraks)$.

\begin{prop}
\label{obatruction at the level of Diff}
Let $M$ be a closed smooth $n$-manifold with $n \geq 1$, $(X, \fraks)$ be a closed smooth ${\it spin}^{c}$ $4$-manifold with $d(\fraks) = -n$ and $b^{+}(X) \geq n + 2$, $\Phi : \pi_{1}(M) \to \Diff(X, \fraks)$ be a group homomorphism, and $\rho : M \to B\Diff(X, \fraks)$ be the classifying map of the bundle $X \to E_{X} \to M$ given  by the Borel construction with respect to the actions of $f_{1}, \ldots, f_{n}$ on $X$.
\begin{description}
\item[(i)]
\label{case that without homology ori}
Assume that $\SWinv(\Phi; \fraks) \neq 0$ in $\Z/2$.
Then, for any $(n+1)$-dimensional compact manifold $W$ with $\del W = M$, there exists no continuous map $\tilde{\rho} : W \to B\Diff(X, \fraks)$ such that the following diagram commutes:
\begin{align*}
\xymatrix{
    W \ar[dr]^{\tilde{\rho}} &  \\
    M \ar[r]_-{\rho} \ar@{}[u]|-*{\rotatebox{90}{$\subset$}} &  B\Diff(X, \fraks).
    }
\end{align*}
\item[(ii)]
\label{case that with homology ori}
Assume that $\im{\Phi} \subset \Diff(X, \fraks, \mathcal{O})$, $M$ is oriented, and that $\SWinv(\Phi; \fraks) \neq 0$ in $\Z$.
Let us regard $\rho$ as a map into $B\Diff(X, \fraks, \mathcal{O})$.
Then, for any $(n+1)$-dimensional compact oriented manifold $W$ with $\del W = M$, there exists no continuous map $\tilde{\rho} : W \to B\Diff(X, \fraks, \mathcal{O})$ such that the following diagram commutes:
\begin{align*}
\xymatrix{
    W \ar[dr]^{\tilde{\rho}} &  \\
    M \ar[r]_-{\rho} \ar@{}[u]|-*{\rotatebox{90}{$\subset$}} &  B\Diff(X, \fraks, \mathcal{O}).
    }
\end{align*}
\end{description}
\end{prop}

\begin{proof}
We give the proof for the case (i); that for the case (ii) is similar to it.
As in \cref{subsection A generalized invariant}, take a model of $B\Diff(X, \fraks)$ to be a smooth infinite dimensional manifold.
Then we may assume that $\rho$ is smooth and it is sufficient to see that there is no smooth $\tilde{\rho}$ which makes the diagram in the statement commuting.
Assume that such a map $\tilde{\rho}$ does exist.
Then we obtain a smooth bundle $E^{W}_{X} \to W$ whose fiber is $X$ with structure group $\Diff(X, \fraks)$ and also obtain $E_{\circPi}^{W} \to W$ whose fiber is $\circPi(X)$ satisfying that $(E_{X}^{W})|_{\del W} = E_{X}$ and that $(E_{\circPi}^{W})|_{\del W} = E_{\circPi}$.
Take a generic section $s : \del W \to E_{\circPi}$ and its generic extension $\tilde{s} : W \to E_{\circPi}^{W}$.
By the same argument in \cref{subsection A generalized invariant}, we can consider the moduli space $\calM(\rho, s, \fraks)$ and $\calM(\tilde{\rho}, \tilde{s}, \fraks)$ parameterized on $M$ and on $W$ respectively.
The moduli space $\calM(\tilde{\rho}, \tilde{s}, \fraks)$ is a $1$-dimensional compact manifold with boundary $\calM(\rho, s, \fraks)$, and therefore we obtain $\#\calM(\rho, s, \fraks) = 0$ in $\Z/2$.
This contradicts the assumption $\SWinv(\Phi; \fraks) \neq 0$.
\end{proof}

Using \cref{obatruction at the level of Diff} and the non-vanishing theorem for our invariant (\cref{non-vanishing theorem}), in \cref{extension problem at the level of Diff} we will give an example of a family on $T^{n}$ of $4$-manifolds with structure group $\Diff(X, \fraks)$ which cannot be extended to an $(n+1)$-dimensional manifold $W$ bounded by $T^{n}$.

\section{Non-vanishing and applications}
\label{Section Non-vanishing and applications}

In this section we show a non-vanishing theorem for our invariant and give some applications, in particular the proof of \cref{application to PSC}.
The mechanism of the non-vanishing is quite similar to that of the cohomological \SW invariant introduced in \cite{Konno2} by the author.
There are two key tools to prove the non-vanishing:
the first one is the combination of wall-crossing and gluing technique due to Ruberman~\cite{MR1671187, MR1734421, MR1874146}, and the second is the description of higher-dimensional wall-crossing phenomena in terms of embedded surfaces given in \cite{Konno1} by the author.
\Cref{subsection: description of higher-dimensional wall-crossing phenomena in terms of embedded surfaces} is used to adjust the second tool in the situation of this paper.
In \cref{subsection: ex. of simple invariant} we give the concrete examples of diffeomorphisms for which our invariant does not vanish and an application of the non-vanishing to PSC metrics.
We also use the non-vanishing to give an obstruction to extensions of families of $4$-manifolds with structure group $\Diff(X, \fraks)$ in \cref{subsection: Application to the extension problem for families of 4manifolds}.

\subsection{Description of higher-dimensional wall-crossing}
\label{subsection: description of higher-dimensional wall-crossing phenomena in terms of embedded surfaces}

In~\cite{Konno1}, the author has given a description of higher-dimensional wall-crossing phenomena in terms of embedded surfaces.
In this subsection we recall and rewrite a part of it in a convenient form to prove the non-vanishing result for our invariant.
For natural numbers $n \geq 1$ and $m > n$, set $N = n\CP^2 \# m(-\CP^2) = \#_{i=1}^n \CP^2_i \# (\#_{j=1}^m (-\CP^2_j))$ and
let $H_i$ and $E_j$ be a generator of $H^2(\CP^2_i)$ and one of $H^2(-\CP^2_j)$ respectively.
Let $\frakt$ be the $\spc$ structure on $N$ satisfying that $c_1(\frakt) = \sum_{i=1}^n H_i + \sum_{j=1}^m E_j$.

Let $\Si_i^{+}, \Si_i^{-} \inc N$ $(i = 1, \ldots, n)$ be oriented closed connected surfaces embedded in $N$ such that 
\begin{itemize}
\item $[\Si_i^{\pm}]^2 = 0$,
\item for each $i$, $\Si_i^+$ and $\Si_i^-$ intersect transversally,
\item if $i \neq j$, then $\Si_i^{\epsilon_{i}} \cap \Si_j^{\epsilon_{j}} = \emptyset$ holds for any $\epsilon_i, \epsilon_j \in \{+, -\}$.
\end{itemize}
Take a small closed tubular neighborhood $U(\Si_{i}^{\pm})$ of each $\Si_{i}^{\pm}$.
Since the normal bundle of $\Si_i^{\pm}$ is trivial, $U(\Si_{i}^{\pm})$ is diffeomorphic to $D^2 \times \Si_i^{\pm}$.
Let $V(\Si_{i}^{\pm})$ be a closed tubular neighborhood of $\del U(\Si_{i}^{\pm}) \cong S^{1} \times \Si_{i}^{\pm}$, which is diffeomorphic to $[0,1] \times S^{1} \times \Si_{i}^{\pm}$.
We may assume that $V(\Si_{i}^{\epsilon_{i}})$ and $V(\Si_{j}^{\epsilon_{j}})$ are disjoint for any $i, j$ with $i \neq j$ and any $\epsilon_i, \epsilon_j \in \{+, -\}$ by taking $U(\Si_{i}^{\pm})$ and $V(\Si_{i}^{\pm})$ to be sufficiently small.
Given a positive number $R$, let us consider a map $\phi_{N} : [0,1]^{n} \to \Pi(N)$ satisfying

\begin{itemize}
\item $\phi_{N}(\del[0,1]^{n}) \subset \Met(N)$,
\item for any $j \in \{1, \ldots,n\}$ and for any metric $g \in \phi_{N}(F_{0, j}^{n-1})$, the Riemannian submanifold  $(V(\Si_{j}^{-}), g)$ is isometric to $[0, R] \times S^{1} \times \Si_{j}^{-}$, and
\item for any $j \in \{1, \ldots,n\}$ and for any metric $g \in \phi_{N}(F_{1, j}^{n-1})$, the Riemannian submanifold  $(V(\Si_{j}^{+}), g)$ is isometric to $[0, R] \times S^{1} \times \Si_{j}^{+}$.
\item
Assume that the family $\phi_N : [0, 1]^{n} \to \Pi(N)$ satisfies that $\phi_N(\del[0, 1]^{n}) \subset \circPi(N)$.
\end{itemize}
Here, in the  product $[0, R] \times S^{1} \times \Si_{j}^{\pm}$, we equip $\Si_{j}^{\pm}$ with a metric of constant scalar curvature and of unit area and $S^{1}$ with the metric of unit length, and $F^{n-1}_{i,j}$ $(i=0,1)$ are the faces of $[0,1]^{n}$ given as \eqref{eq: facet time-0} and \eqref{eq: facet time-1}.
Then one can define the ``intersection number'' $\phi_N \cdot \calW(N)$.
This intersection number can be interpreted as the mapping degree of the map $\phi_{N}|_{\del[0, 1]^{n}} : \del[0, 1]^{n} \to \circPi(N) \simeq S^{n-1}$. 
Here the given orientation of $H^+(N; \R)$ is used to determine the sign of the mapping degree, however we will not specify the signs it since we will work on $\Z/2$ in the proof of the non-vanishing theorem.

\begin{prop}[See Lemma 3.3~\cite{Konno1}]
\label{prop: description of higher-dim. wall-crossing}
Suppose that $H_j \cdot [\Si_i^{\pm}] = 0$ holds for any $i, j \in \{1, \ldots, n\}$ with $i \neq j$, and the two integers
\[
(c_1(\frakt) \cdot  [\Si_i^{+}]) \cdot (H_i \cdot [\Si_i^{+}])
\]
and
\[
(c_1(\frakt) \cdot  [\Si_i^{-}]) \cdot (H_i \cdot [\Si_i^{-}])
\]
are non-zero and have the different signs for each $i$.
If $R$ is sufficiently large, then the family $\phi_N : [0, 1]^{n} \to \Pi(N)$ satisfies that $\phi_{N}(\del[0,1]^{n}) \subset \circPi(N)$ and that $\phi_N \cdot \calW(N) = \pm1$.
\end{prop}

\begin{proof}
This \lcnamecref{prop: description of higher-dim. wall-crossing} follows from the argument of Section~3 in \cite{Konno1}.
The key observation is that, in Lemma~3.2 in \cite{Konno1}, we do not need to assume that metrics in the statement of the lemma, containing the cylindrical part $[0, R_{i}] \times S^{1} \times \Si_{i}$, are obtained from the stretching construction starting from a given initial metric.
(See the proof of Lemma~3.2 in \cite{Konno1}.)

The detailed way to adapt the argument of Section~3 in \cite{Konno1} for the current situation is as follows.
The first step is to see that $\Si_{1}^{\pm}, \ldots, \Si_{n}^{\pm}, c_{1}(\frakt)$ satisfy Condition~1 in page~1138 of \cite{Konno1}.
In the proof of Corollary~2.15 in \cite{Konno1}, it is shown that $\Si_{1}^{\pm},\ldots, \Si_{b^{+}}^{\pm}, c'$ satisfy Condition~1 using the fact that $\alpha_{1}^{\pm},\ldots, \alpha_{b^{+}}^{\pm}, c'$ satisfy the equation (10), where $\alpha_{i}^{\pm}=[\Si_{i}^{\pm}]$.
The proof that $\Si_{1}^{\pm}, \ldots, \Si_{n}^{\pm}, c_{1}(\frakt)$ satisfy Condition~1 is completely same with this argument in the proof of Corollary~2.15.

The second step is to see that, if $\Si_{1}^{\pm}, \ldots, \Si_{n}^{\pm}, c_{1}(\frakt)$ satisfy Condition~1, then the wall-crossing happens, namely, the mapping degree of $\phi_N|_{\del [0,1]^{n}} : \del[0, 1]^{n} \to \circPi(N)$ is $\pm1$.
Lemma~3.3 in \cite{Konno1} provides the proof of this fact.

Here we give some comments to adapt Lemma~3.3.
In \cite{Konno1}, the map $\mathcal{F} : \mathcal{P} \to (V^{+})^{\ast}$ defined in page~1144  of \cite{Konno1} describes the wall-crossing.
Here $\mathcal{P}$ denotes the parameter space, corresponding to $[0,1]^{n}$ in the current situation.
The codomain of this map $\mathcal{F}$ is a $b^{+}$-dimensional vector space $(V^{+})^{*}$, not $\circPi(N)$.
However, the spaces $\circPi(N)$, $\Met(N) \cap \circPi(N)$, and $(V^{+})^{\ast} \setminus \{0\}$ are mutually homotopy equivalent.
(A homotopy equivalence map between $\Met(N) \cap \circPi(N)$ and $(V^{+})^{\ast} \setminus \{0\}$ is given by the restriction of the map (13) in \cite{Konno1}.)
Therefore we can rewrite the statement of Lemma~3.3 to be a statement on the mapping degree of a map into $\circPi(N)$ rather than $(V^{+})^{\ast} \setminus \{0\}$.
The rest difference between the appearance of the current situation and that of Lemma~3.3 is, although \cref{prop: description of higher-dim. wall-crossing} claims that $\phi_N \cdot \calW(N) = \pm1$, Lemma~3.3 in \cite{Konno1} claims that the mapping degree of $\mathcal{F}$ is just non-zero.
This difference comes from condition (ii) of Condition~1 says that the mapping degree of $F$ is just non-zero.
However, in the current situation, one can check that the conclusion of this condition (ii) can be replaced with the statement that the mapping degree of $F$ is $\pm1$.
This follows from the proof of Corollary~2.15 in \cite{Konno1} again.
\end{proof}

\subsection{Non-vanishing theorem and the proof of \cref{application to PSC}}
\label{subsection: ex. of simple invariant}

In this subsection we give the non-vanishing theorem for our invariant and prove \cref{application to PSC}.

\begin{theo}
\label{non-vanishing theorem}
Let $n > 0$ and set $N = n\CP^{2}\#2n(-\CP^{2})$.
Let $\frakt$ be the ${\it spin}^{c}$ structure on $N$ with $c_{1}(\frakt) =  \sum_{i=1}^n H_i + \sum_{j=1}^{2n} E_j$ and $(M, \fraks_{0})$ be a closed smooth ${\it spin}^{c}$ $4$-manifold with $b^{+}(M) \geq 2$ and $d(\fraks_{0}) = 0$.
Set $(X, \fraks) = (M \# N, \fraks_{0} \# \frakt)$.
Then, there exist $f_{1}, \ldots, f_{n} \in \Diff(X, \fraks)$ satisfying:
\begin{itemize}
\item $f_{1}, \ldots, f_{n}$ are commuting,
\item all of $f_{1}, \ldots, f_{n}$ reverse a given homology orientation of $X$, and
\item $\SWinv(f_{1}, \ldots, f_{n}; \fraks) = \SWinv(M, \fraks_{0})$ holds in $\Z/2$, where $\SWinv(M, \fraks_0)$ denotes the \SW invariant of $(M, \fraks_0)$.
\end{itemize}
\end{theo}

\begin{proof}
Note that $d(\fraks) = -n$ and $b^{+}(X) \geq n+2$.
We write $N$ as
\[
N = \#_{i=1}^{n}\left(\CP_{i}^{2} \# (-\CP^{2}_{i, 1}) \# (-\CP^{2}_{i, 2})\right)
\]
and set $N_{i} = \CP_{i}^{2} \# (-\CP^{2}_{i, 1}) \# (-\CP^{2}_{i, 2})$.
The ordered basis $\{ H_{1}, \ldots, H_{n}\}$ of $H^{+}(N)$ gives the homology orientation of $N$, and by fixing a homology orientation of $M$, we obtain one of $X$.

Let  $\frakt_{0}$ be the $\spc$ structure on $N_{0} = \CP^{2}\#(-\CP^{2}_{1})\#(-\CP^{2}_{2})$ with $c_{1}(\frakt_{0}) = H+E_{1}+E_{2}$, where $H,$ $E_{1}$, and $ E_{2}$ are generators of $H^{2}(\CP^{2})$, $H^{2}(-\CP^{2}_{1})$, and $H^{2}(-\CP^{2}_{2})$ respectively.
Let us identify the second cohomology group with the second homology group via Poincar\'{e} duality.
Take a sphere $S$ which represents $c_{1}(\frakt_{0})$.
For this $(-1)$-curve, let $\rho_{S} : N_{0} \to N_{0}$ be a diffeomorphism which gives rise to the reflection with respect to $S$ on cohomology: the induced map is $\rho_{S}^{\ast}(\alpha) = \alpha + 2(\alpha \cdot [S])[S]$ for $\alpha \in H^{2}(N_{0})$.
(The use of this reflection is due to Ruberman~\cite{MR1671187, MR1734421, MR1874146}.)
We may take $\rho_{S}$ to be the identity on a disk in $N_{0}$.
Let us also take embedded spheres which represent $H$ and $E_{j}$ respectively, and similarly define diffeomorphisms $\rho_{H}$ and $\rho_{E_{j}}$ on $\CP^{2}$ and $-\CP^{2}_{j}$ respectively.
(We note that the sign in the formula for $\rho_{H}^{\ast}$ is changed since $H$ is a $(+1)$-curve:
$\rho_{H}^{\ast}(\alpha) = \alpha - 2(\alpha \cdot H)H$.)
The maps $\rho_{H}^{\ast}$ and $\rho_{E_{j}}^{\ast}$ induce just the multiplication by $-1$ on $H^{2}(\CP^{2})$ and $H^{2}(-\CP^{2}_{j})$ respectively.
We define a diffeomorphism $f'_{0}$ on $N_{0}$ by $f'_{0} := (\rho_{H} \# \rho_{E_{1}} \# \rho_{E_{2}}) \circ \rho_{S}$.
The induced map $(f'_{0})^{\ast}$ on cohomology is expressed as the matrix
\[
\begin{pmatrix}
-3 & 2 & 2\\
-2 & 1 & 2\\
-2 & 2 &1
\end{pmatrix},
\]
where we take $\{H, E_{1}, E_{2}\}$ as an ordered basis of $H^{2}(N_{0})$.
From this description, it is easy to see that $f_{0}'$ preserves $\frakt_{0}$ (at the level of isomorphism classes) and the orientation of $N_{0}$, and reverses an orientation of $H^{+}(N_{0})$.
We take an embedded sphere $\Si_{0}^{+}$ in $N_{0}$ which represents $H-E_{1}$, and
set $\Si_{0}^{-} := f_{0}'(\Si_{0}^{+})$.
By perturbing $f_{0}'$ using isotopy, we can arrange that $\Si_{0}^{+}$ and $\Si_{0}^{-}$ intersect transversely.
As in \cref{subsection: description of higher-dimensional wall-crossing phenomena in terms of embedded surfaces}, let $U(\Si_{0}^{-})$ be a closed tubular neighborhood of $\Si_{0}^{-}$ in $N_{0}$ and $V(\Si_{0}^{-})$ be a closed tubular neighborhood of $\del U(\Si_{0}^{-})$.
We also take a metric $g_{0} \in \Met(N_{0})$ such that the Riemannian submanifold $(V(\Si_{0}^{-}), g_{0})$ is isometric to $[0, R] \times S^{1} \times \Si_{0}^{-}$ for a sufficiently large $R$.
Note that $((f_{0}')^{-1}(V(\Si_{0}^{-})), (f'_{0})^{\ast}g_{0})$ is isometric to $[0, R] \times S^{1} \times \Si_{0}^{+}$.
Let $\mu_{0} \in \Omega^{+}_{g_{0}}(N_{0})$ be a small generic self-dual $2$-form whose support is contained in a small open set in $N_{0}$ and $\phi_{0} : [0,1] \to \Pi(N_{0})$ be a generic path from $(g_{0}, \mu_{0})$ to $(f'_{0})^{\ast}(g_{0}, \mu_{0})$.

For each $i \in \{1, \ldots, n\}$, we define a diffeomorphism $f'_{i}$ on $N_{i}$ as the copy of $f'_{0}$.
Let $\Si_{i}^{\epsilon}$ be the surface embedded in $N_{i}$ given as the copy of $\Si_{0}^{\epsilon}$.
We also regard each $\Si_{i}^{\epsilon}$ as an embedded surface in $N$ and in $X$.
It can be seen easily that the collection $\{\Si_i^{\epsilon}\}_{1 \leq i \leq n, \epsilon \in \{+,-\}}$ satisfies the assumption of \cref{prop: description of higher-dim. wall-crossing}.
Let $\phi_{i} : [0,1] \to \Pi(N_{i})$ be the copy of $\phi_{0}$, and we define a map $\phi_{N} : [0,1]^{n} \to \Pi(N)$ by
\[
\phi_{N}(t_{1}, \ldots, t_{n}) := \phi_{1}(t_{1}) \# \cdots \# \phi_{n}(t_{n}),
\]
where the connected sum is considered on the complement of the supports of the copies of $\mu_{0}$.
The mapping degree of $\phi_{N}|_{\del[0, 1]^{n}} : \del[0, 1]^{n} \to \circPi(N) \simeq S^{n-1}$ is invariant under a small perturbation, and therefore we obtain $\phi_N \cdot \calW(N) = \pm1$ from \cref{prop: description of higher-dim. wall-crossing}.

Let $f_{i}$ be the diffeomorphism on $X$ defined as the connected sum of $f'_{i}$ and the identity map on $M\#(\#_{\substack{1 \leq j \leq n,\\ j \neq i }} N_{j})$.
Each $f_{i}$ preserves $\fraks$ and reverses the homology orientation of $X$.
Since $f_{1}, \ldots, f_{n}$ are obviously commuting, we can consider $\SWinv(f_1, \ldots, f_n; \fraks) \in \Z/2$.
Let $(g_{M}, \mu_{M}) \in \Pi(M)$ be a generic point and $B_{M}^{4} \subset M$ and $B_{N}^{4} \subset N$ be small balls.
We may also assume that $\mu_{M}$ is supported on the complement of $B_{M}^{4}$, and $B_{N}^{4}$ is contained in the complement of the supports of the copies of $\mu_{0}$.
Then we can define $\phi : [0,1]^{n} \to \circPi(X)$ by $\phi(x) := (g_{M}, \mu_{M})\#\phi_{N}(x)$, where $\#$ is the connected sum consisting of $M \setminus B^{4}_{M}$, $N \setminus B^{4}_{N}$, and a cylinder with the standard product metric and of sufficiently large length.
Since we have $d(\fraks) = -n$, the moduli space $\calM(\phi, \fraks)$ parameterized by $\phi$ on $[0, 1]^{n}$ is a $0$-dimensional compact manifold.
For this moduli space,
\begin{align}
\# \calM(\phi, \fraks) = \pm(\phi_{N} \cdot \calW(N)) \cdot \SWinv(M, \fraks_0)
= \pm\SWinv(M, \fraks_0)
\label{calculation of moduli and equality to SW inv}
\end{align}
holds in $\Z$ by Ruberman's combination of wall-crossing and gluing arguments~\cite{MR1671187, MR1734421, MR1874146}.
Note that the sign may change in the last equality.

\begin{rem}
Ruberman's combination of wall-crossing and gluing arguments is summarized as Proposition 4.1 in \cite{Konno2}.
After writing this paper, Baraglia and the author~\cite{BK} generalized Proposition 4.1 in \cite{Konno2} to more general families of 4-manifolds.
\end{rem}

For distinct indices $i_1, \ldots, i_k \in \{1, \ldots, n\}$ with $i_1< \cdots < i_k$, we define a smooth generic map
\[
\vp_k(f_{i_1}, \ldots, f_{i_k}) : [0,1]^k \to \circPi(X)
\]
as the composition $\phi \circ \iota_{i_{1}, \ldots, i_{k}} : [0, 1]^{k} \to \circPi(X)$, where
\[
\iota_{i_{1}, \ldots, i_{k}} : [0, 1]^{k} \cong [0, 1]^{\{i_{1}, \ldots, i_{k}\}} \to [0, 1]^{n}
\]
is the inclusion corresponding to the coordinates $t_{i_{1}}, \ldots, t_{i_{k}}$.
Of course, the map $\vp_n(f_{1}, \ldots, f_{n})$ coincides with $\phi$.
Note that we have
\[
f_{i_{1}}^{\ast} \cdots f_{i_{k}}^{\ast}(\vp_{0}(0)) = \phi(v_{i_{1}, \ldots, i_{k}}),
\]
where $v_{i_{1}, \ldots, i_{k}} \in [0, 1]^{n}$ is the vector whose $i$-th component is $0$ if $i \in \{1, \ldots, n\} \setminus \{i_{1}, \ldots, i_{k}\}$ and is $1$ if $i \in \{i_{1}, \ldots, i_{k}\}$.
This $\vp_{\bullet}$ can be therefore used to calculate $\SWinv(f_{1}, \ldots, f_{n}; \fraks)$.
Namely, we have
\[
\SWinv(f_{1}, \ldots, f_{n}; \fraks)
= \SWinv(f_1, \ldots, f_n; \fraks; \vp_{\bullet})
= \# \calM(\vp_n(f_{1}, \ldots, f_{n}), \fraks)
= \# \calM(\phi, \fraks).
\]
From this equality and \eqref{calculation of moduli and equality to SW inv}, we obtain
\[
\SWinv(f_{1}, \ldots, f_{n}; \fraks) = \SWinv(M, \fraks_{0})
\]
in $\Z/2$.
This proves the \lcnamecref{non-vanishing theorem}.
\end{proof}

\begin{cor}
\label{cor application to PSC}
Let $(M, \fraks_{0})$ and $(X, \fraks)$ be the ${\it spin}^{c}$ $4$-manifolds given in the statement of \cref{non-vanishing theorem}.
Suppose that $\SWinv(M, \fraks_{0}) = 1$ in $\Z/2$ and $\PSC(X) \neq \emptyset$.
Then, 
\[
\pi_{i}(\PSC(X)) \neq 0
\]
holds for at least one $i \in \{0, \ldots, n-1\}$.
\end{cor}

\begin{proof}
Let $f_{1}, \ldots, f_{n}$ be the diffeomorphisms whose existence is assured by \cref{non-vanishing theorem}.
Since we have $\SWinv(f_1, \ldots, f_n; \fraks) = \SWinv(M, \fraks_0) \neq 0$ in $\Z/2$, the assertion of the \lcnamecref{cor application to PSC} follows from \cref{vanishing of homotopy groups and invariant}.
\end{proof}

\begin{proof}[Proof of \cref{application to PSC}]
As in the proof of Corollary 5.2 in Ruberman~\cite{MR1874146}, let $M$ be the blowup at $l-2n+2$ points of the elliptic surface with $b^{+}=2k-1$, and $\fraks_{0}$ be the $\spc$ structure on $M$ which arises from the complex structure.
This $\spc$ $4$-manifold $(M, \fraks_{0})$ satisfies the assumption in \cref{non-vanishing theorem}.
Note that $M \# N$ is diffeomorphic to $X$ given in the statement of the \lcnamecref{application to PSC} (see \cite{MR537731,MR0491730}).
We have therefore proven the \lcnamecref{application to PSC} from \cref{cor application to PSC}.
\end{proof}

\begin{rem}
\label{rem: comparison with Ruberman and Xu}
We here compare \cref{application to PSC} with Ruberman's result~\cite{MR1874146} and Xu's~\cite{MR2706507} in detail.
Ruberman proved $\pi_{0}(\PSC(X)) \neq 0$ for $X = 2k\CP^{2} \# l(-\CP^{2})$ with $k \geq 2$ and a sufficiently large $l$, given as Corollary~5.2 in \cite{MR1874146}.
For $4$-manifolds with odd $b^{+}$, 
the result due to Xu gives the non-triviality of $\pi_{0}(\PSC(X))$ for $X = (4k+7)\CP^{2} \# l(-\CP^{2})$ with $k \geq 0$ and a sufficiently large $l$.
(See comments after Theorem~33 in \cite{MR2706507}.
Xu has considered the connected sum of two copies of an algebraic surface $X'$ with $b^{+}(X') \equiv 3$ mod $4$ there.
From this the non-triviality of $\pi_{0}(\PSC(X))$ is deduced for $X = (8k+7)\CP^{2} \# l(-\CP^{2})$ with $k \geq 0$.
If we consider the connected sum of $X'$ and $K3$, one can show that the non-triviality for $(4k+7)\CP^{2} \# l(-\CP^{2})$ by Theorem~33 in \cite{MR2706507} and Bauer's product formula \cite{MR2025299} for the Bauer--Furuta invariant.)

Let us return to our \cref{application to PSC}.
In the case that $n$ is odd, Ruberman's result is stronger than that of \cref{application to PSC} in general.
We note that, in the case that $n=1$, the proof of \cref{application to PSC} can be regarded as an alternative proof of the result due to Ruberman on the disconnectivity of $\PSC(X)$ for $X = 2k\CP^{2} \# l(-\CP^{2})$ using the invariant defined in \cite{MR1671187} rather than the invariant $\SWinv_{tot}$ defined in \cite{MR1874146}.
(In \cite{MR1874146}, Ruberman has showed not only the disconnectivity but also $\pi_{0}(\PSC(X))$ is infinite using $\SWinv_{tot}$.)
We next consider the case that $n$ is even; set $n=2m$ $(m \geq 1)$.
If $k+m-4$ is even, Xu's result is stronger than that of \cref{application to PSC}.
The new part of the result of \cref{application to PSC} is the case that $k+m-4$ is odd;
this case cannot be deduced from Ruberman's result and Xu's.
\Cref{application to PSC} therefore provides new constraints on $\PSC(X)$ for infinitely many $4$-manifolds $X$'s having distinct $b^{+}$.
\end{rem}

\begin{rem}
At this stage the author does not know which $i \in \{0, \ldots, n-1\}$ satisfies $\pi_{i}(\PSC(X)) \neq 0$ in \cref{application to PSC} for $n>1$.
We therefore present the following question:

\begin{problem}
\label{question}
In the setting of \cref{application to PSC} with $n>1$, detect $i \in \{0, \ldots, n-1\}$ satisfying that $\pi_{i}(\PSC(X)) \neq 0$.
\end{problem}

As explained in the introduction, \cref{application to PSC} seems the best possible constraint on $\PSC(X)$ obtained from the direct generalization of Ruberman's argument in~\cite{MR1671187, MR1734421, MR1874146}: the higher-dimensional wall-crossing.
One therefore needs another technique to attack \cref{question}.
Although it is, of course, a difficult problem to show some vanishing/non-vanishing result for homotopy groups of $\PSC(X)$ in general,
the author expects that some combination of the invariant defined in this paper and ideas given in Auckly--Kim--Melvin--Ruberman~\cite{MR3355110} provides a way to approach it.
\end{rem}

\subsection{Non-extendable families of $4$-manifolds}
\label{subsection: Application to the extension problem for families of 4manifolds}

In \cref{subsection: Invariant as an obstruction}, we have mentioned that our invariant can be interpreted as an obstruction to extensions of families of $4$-manifolds.
As a corollary of \cref{non-vanishing theorem}, we can exhibit a family of $4$-manifolds which is obstructed by our invariant:
we give an example of families on $T^{n}$ of $4$-manifolds with structure group $\Diff(X, \fraks)$ which cannot be extended to certain larger base spaces.

We fix $n \geq 1$.
Let $(M, \fraks_{0})$, $(X, \fraks)$ and $f_{1}, \ldots, f_{n} \in \Diff(X, \fraks)$ be the $\spc$ $4$-manifolds and the diffeomorphisms given in the statement of \cref{non-vanishing theorem}, where we
take $(M, \fraks_{0})$ to be $\SWinv(M, \fraks_{0}) = 1$ in $\Z/2$.
For example, we can take $(2k + n - 1)\CP^{2} \# l(-\CP^{2})$ as $X$ for $k \geq 2$ and $l \geq 10k + 2n -1$ as in the proof of \cref{application to PSC}, and on the other hand we can also take a non-simply connected $4$-manifold as $X$.
Let $\Phi : \left<f_{1}, \ldots, f_{n} \right> \inc \Diff(X, \fraks)$ be the inclusion and $\rho : T^{n} \to B\Diff(X, \fraks)$ be the classifying map of the bundle $E_{X} \to T^{n}$ given  by the Borel construction with respect to the actions of $f_{1}, \ldots, f_{n}$ on $X$.

\begin{cor}
\label{extension problem at the level of Diff}
The ${\it spin}^{c}$ $4$-manifold $(X, \fraks)$ and the map $\rho : T^{n} \to B\Diff(X, \fraks)$ given above satisfy the following property:
For any $(n+1)$-dimensional compact smooth manifold $W$ with $\del W = T^{n}$, there exists no continuous map $\tilde{\rho} : W \to B\Diff(X, \fraks)$ such that the following diagram commutes:
\begin{align}
\xymatrix{
    W \ar[dr]^{\tilde{\rho}} &  \\
    T^{n} \ar[r]_-{\rho} \ar@{}[u]|-*{\rotatebox{90}{$\subset$}} &  B\Diff(X, \fraks).
    }
\label{diagram classifying space}
\end{align}
\end{cor}

\begin{proof}
By \cref{interpretation of SW for tuples} and  \cref{non-vanishing theorem}, we have
\[
\SWinv(\Phi; \fraks)
= \SWinv(f_{1}, \ldots, f_{n} ;\fraks) = \SWinv(M, \fraks_{0}) = 1
\]
in $\Z/2$.
The assertion therefore follows from \cref{obatruction at the level of Diff}.
\end{proof}

\begin{rem}
\label{rem extension and ex for surface bundle}
At this stage, the author does not know whether the map $\rho : T^{n} \to B\Diff(X, \fraks)$ in \cref{extension problem at the level of Diff} cannot be extended to a map from $W$ to $B\Diff^{+}(X)$ for any $W$.
This is a non-trivial question, though the family corresponding to $\rho$ given as a mapping torus obtained from concrete commuting diffeomorphisms $f_{1}, \ldots, f_{n}$, and the action of $f_{i}$ on $H_{\ast}(X ;\Z)$ is non-trivial for any $i$:
in general, for the mapping torus obtained from given commuting diffeomorphisms on a given manifold, the non-triviality of the action of the diffeomorphisms on the homology group of the fiber is not sufficient to deduce that the mapping torus cannot be extended as a $\Diff^{+}(X)$-bundle to a given new base space $W$ bounded by the torus, which is the original base space.

One can check this even for surface bundles, rather than $4$-manifold bundles.
To be precise, for any $g,g' \geq 1$, there exists an orientation preserving diffeomorphism $f$ on an oriented closed surface $\Si_{g}$ of genus $g$ such that $f$ acts on $H^{1}(\Si_{g})$ non-trivially, but the mapping torus $\Si_{g} \to E_{\Si_{g}} \to S^{1}$ with respect to $f$ can be extended to a bundle over $W=\Si_{g',1}$, an oriented compact surface of genus $g'$ with one boundary component $S^{1}$, which is regarded as the base space of $E_{\Si_{g}}$.

The construction of such $f$ is as follows.
Firstly, note that we can take 
\[
f_{1}, f_{1}', \ldots, f_{g'}, f_{g'}' \in \Diff^{+}(\Si_{g})
\]
to be $(\prod_{i=1}^{g'}[f_{i},f_{i}'])^{*} : H^{1}(\Si_{g}) \to H^{1}(\Si_{g})$ is non-trivial.
To get such $f_{1}, f_{1}', \ldots, f_{g'}, f_{g'}'$, take
$A_{1}, A_{2} \in Sp(2,\Z)$ with $[A_{1}, A_{2}] \neq 1$, for example
\begin{align*}
A_{1}=
\begin{pmatrix}
2 & 1\\
2 & 1
\end{pmatrix},
A_{2}=
\begin{pmatrix}
1 & 1\\
1 & 2
\end{pmatrix}.
\end{align*}
Recall the fact that the natural map $p : \Diff^{+}(\Si_{g}) \to Sp(2g,\Z)$ is surjective.
Therefore we can find $f_{1}, f_{2}$ satisfying
$p(f_{i})=A_{i}$, where precisely this $A_{i}$ means the image of $A_{i} \in Sp(2,\Z)$ into $Sp(2g,\Z)$ by the natural inclusion $Sp(2,\Z) \subset Sp(2g,\Z)$.
Then, by defining $f_{2}, f_{2}', \ldots, f_{g'}, f_{g'}'=1$, we have that $(\prod_{i=1}^{g'}[f_{i},f_{i}'])^{*} : H^{1}(\Si_{g}) \to H^{1}(\Si_{g})$ is non-trivial.
Set $f = \prod_{i=1}^{g'}[f_{i},f_{i}']$.

Secondly, we shall show that the mapping torus $\Si_{g} \to E_{\Si_{g}} \to S^{1}$ with respect to $f$ can be extended to a bundle over $W=\Si_{g',1}$.
Recall the standard representation of $\pi_{1}(\Si_{g',1})$:
\[
\pi_{1}(\Si_{g',1}) = \left<A_{1}, B_{1}, \ldots, A_{g'},B_{g'},C \mid \prod_{i=1}^{g'}[A_{i},B_{i}]=C \right>.
\]
Using $f_{1}, f_{1}', \ldots, f_{g'}, f_{g'}' \in \Diff^{+}(\Si_{g})$ taken above, 
we get a homomorphism $\tilde{\rho} : \pi_{1}(\Si_{g',1}) \to \Diff^{+}(\Si_{g})$
by
$\tilde{\rho}(A_{i})=f_{i}, \tilde{\rho}(B_{i})=f_{i}'$, and $\tilde{\rho}(C)= f$.
Thus we have a bundle $\Si_{g} \to \tilde{E}_{\Si_{g}} \to \Si_{g',1}$ by the Borel construction via $\tilde{\rho}$.
The restriction $\Si_{g} \to \tilde{E}_{\Si_{g}}|_{\del \Si_{g',1}} \to \del \Si_{g',1} = S^{1}$ coincides with the mapping torus with respect to $f$.
Therefore the map $f$ satisfies the desired conditions.
\end{rem}

Using \cref{extension problem at the level of Diff}, we can study a purely group theoretic property on $\Diff(X, \fraks)$.

\begin{cor}
\label{extension problem at the level of Diff in terms of groups}
The ${\it spin}^{c}$ $4$-manifold $(X, \fraks)$ and the diffeomorphism $f_{1}, \ldots, f_{n} \in \Diff(X, \fraks)$ given above satisfy the following property:
Let $W$ be an $(n+1)$-dimensional compact smooth manifold with $\del W \cong T^{n}$ and set $G=\pi_{1}(W)$.
Then, there exists no group homomorphism $\phi : G \to \Diff(X, \fraks)$ such that the following diagram commutes:
\begin{align}
\xymatrix{
    G \ar[drr]^{\phi} &  \\
    \Z^{n} \ar[u]^{i_{\ast}} \ar[r] &  \left< f_{1}, \ldots, f_{n} \right> \ar@{}[r]|-*{\subset} &  \Diff(X, \fraks),
    }
\label{diagram groups}
\end{align}
where $i_{\ast}$ is the map induced by the inclusion $i : T^{n} \cong \del W \inc W$.
\end{cor}

\begin{proof}
Assume that there exists $\phi$ which makes the diagram \eqref{diagram groups} commutative.
Let $\tilde{W}$ denotes the universal covering of $W$, and regard $\R^{n}$ as the universal covering of $T^{n}$.
By taking models of $\tilde{W}$ and $\R^{n}$ as path spaces, we can define a map $\R^{n} \to \tilde{W}$ which covers the inclusion $T^{n} \inc W$.
The commutativity of the diagram \eqref{diagram groups} implies that the map $\R^{n} \to \tilde{W}$ induces a well-defined map $\R^{n} \times_{\Z^{n}} X \to \tilde{W} \times_{\pi_{1}(W)} X$, denoted by $F : \R^{n} \times_{\Z^{n}} X \to \tilde{W} \times_{\pi_{1}(W)} X$.
This map $F$ corresponds to a $\Diff(X, \fraks)$-equivariant map $\R^{n} \times_{\Z^{n}} \Diff(X, \fraks) \to \tilde{W} \times_{\pi_{1}(W)} \Diff(X, \fraks)$.
Namely, the map $F$ is a morphism between $\Diff(X, \fraks)$-bundles.
We note that, for each point $p \in T^{n}$ the restriction of $F$ on the fibers on $p$ is invertible, in particular $F$ is injective.
The map $\tilde{\rho}$ defined as the classifying map of $\tilde{W} \times_{\pi_{1}(W)} X$ therefore makes the diagram \eqref{diagram classifying space} commutative.
This contradicts \cref{extension problem at the level of Diff}.
\end{proof}

\begin{rem}
We note that, the assertion of \cref{extension problem at the level of Diff in terms of groups} in the case that $n=1$ follows from an elementary argument below, which is based on the presentation of the fundamental group of a compact surface.
(We note that this argument is obviously valid even if we replace $\Diff(X, \fraks)$ with $\Diff^{+}(X)$.)
If we take $W = D^{2}$, the statement is trivial.
Suppose $W \neq D^{2}$ and there exists $\phi$ which makes the diagram \eqref{diagram groups} commutative.
Then $f$ can be written as $\prod_{j=1}^{g}[h_{j}, h'_{j}]$ or $\prod_{i=1}^{g} h_{j}^{2}$ for some $g \geq 1$ and $h_{j}, h_{j}' \in \Diff(X, \fraks)$.
Since $f$ reverses the homology orientation, this is a contradiction.
\end{rem}

\end{document}